\documentclass[10pt]{amsart}
\usepackage{amssymb,amsmath,amsthm,amsfonts,amsopn,url}
\usepackage{amscd,amssymb,amsopn,amsmath,amsthm,graphics,amsfonts,enumerate,verbatim,calc}
\usepackage[all]{xy}
\theoremstyle{plain}
\newtheorem{thm}{Theorem}[section]
\newtheorem*{mt*}{Main Theorem}
\newtheorem*{cj*}{Conjecture}
\newtheorem*{nt*}{Notations}

\newtheorem{prop}[thm]{Proposition}

\newtheorem{cor}{Corollary}
\newtheorem{rem}{Remark}

\theoremstyle{definition}
\newtheorem{definition}{Definition}
\newtheorem{acknowledgement}{Acknowledgement}

\newcommand{\ideal}[1]{\mathfrak{#1}}

\newcommand{\m}{\ideal{m}}

\newcommand{\p}{\ideal{p}}

\newcommand{\fm}{\frak{m}}

\newcommand{\func}[1]{\mathrm{#1} \,}

\newcommand{\im}{\func{im}}
\newcommand{\Hom}{\func{Hom}}

\newcommand{\RR}{{\mathbb R}}
\newcommand{\Ext}{\operatorname{Ext}}
\newcommand{\cl}{\operatorname{cl}}
\newcommand{\id}{\operatorname{id}}

\title[]{when almost Cohen-Macaulay algebras map into big Cohen-Macaulay modules
}

\author[]{Rajsekhar Bhattacharyya
}

\address{Dinabandhu Andrews College, Garia, Kolkata 700084, India}

\email{rbhattacharyya@gmail.com}

\thanks{}


\begin{document}

\begin{abstract} 
In this paper, we show that almost Cohen-Macaulay algebras are solid. Moreover, we seek for the conditions when (a) an almost Cohen-Macaulay algebra is a phantom extension and (b) when it maps into a balanced big Cohen-Macaulay module.
\end{abstract}

\maketitle

\section{introduction}
Let $(R,\fm)$ be a Noetherian local ring. An $R$-algebra $B$ is defined as big Cohen-Macaulay algebra if some system of parameters of $R$ is a regular sequence on $B$. It is balanced if every system of parameters of $R$ is a regular sequence on $B$. It is a conjecture of Hochster that such algebras exist in general and he proved this for rings that contain a field \cite{Hoc75a}, \cite{Hoc75b}. Recently, in \cite{An}, Andr$\acute{e}$ proved the existence of big Cohen-Macaulay algebras for complete local domain in mixed characteristic $p>0$. 

In \cite{Di1}, a list of seven axioms is proposed to define closure operations for finitely generated modules over a complete local domain $R$. These axioms are also known as "finite axioms" for a closure operation. In a recent work of Dietz \cite{Di2}, these axioms are extended beyond finitely generated modules and we call them as "big axioms". Any closure operation which satisfies these axioms is called Dietz closure. A Dietz closure is powerful enough to produce big Cohen-Macaulay modules. 

We recall the definition of solid algebra and module (see \cite{Ho}) and definitions of phanton extension for a closure operation (see \cite{HH3} and \cite{Di1}). In \cite{Di2}, the notion of phantom extension is extended for closure operation which also satisfies "big axioms" and there it is proved that every phantom extension is solid, but converse is not true.  

Let $R$ be a Noetherian local domain and $R^{+}$ be its integral closure in an algebraic closure of its fraction field. We recall the definition of an almost Cohen-Macaulay algebra $R^{+}$-algebra \cite{Ro} (see Section 2). It is a well-known result due to Hochster that for a complete local domain, any algebra which maps into big Cohen-Macaulay algebra is solid. In this paper, we show that almost Cohen-Macaulay algebras are solid. Moreover, we seek for the conditions when (a) an almost Cohen-Macaulay algebra is a phantom extension and (b) when it maps into a balanced big Cohen-Macaulay module. More precisely: 

\begin{enumerate}
\item[(1)] Over a complete local domain $R$ of mixed characteristic $p>0$, an almost Cohen-Mcaulay $R^+$-algebra $S$, is a solid module. As an immediate consequence of this result, we get that for a complete local domain $R$ with $F$-finite residue field, if $S$ be an almost Cohen-Macaulay $R^+$-algebra, then there always exists a balanced big Cohen-Macaulay module over some complete local domain $R'$ (it may not be a balanced big Cohen-Macaulay $R$-module), such that $R'$ is homomorphic image of $R$ and $S$ is mapped into that balanced big Cohen-Macaulay module (see Theorem 3.1 and Corollary 2).
 
\item[(2)] Over a complete local domain $R$ of mixed characteristic $p>0$, an almost Cohen-Mcaulay $R^+$-algebra $S$ (with certain restrictions on $S$) is a phantom extension of $R$ via the closure defined in Definition 4 (see Section 2) and such an algebra can be modified into a balanced big Cohen-Macaulay module over $R$ (see Theorem 4.1 and Corollary 3).  
\end{enumerate}

\section{preliminary results}

Let $U$ be an integral domain equipped with a mapping $v:U\rightarrow \RR\cup\{\infty\}$, such that for all $a, b\in U$:
\begin{center}
(i) $v(ab) = v(a) + v(b);$\newline
(ii) $v(a + b) \geq \min\{v(a), v(b)\};$\newline
(iii) $v(a) = \infty$ if and only if $a= 0$\newline
\end{center}
We shall refer $v$ as a valuation or a value map. If moreover $v(c)\geq 0$ for every $c\in U$ and $v(c) > 0$ for every non-unit $c\in U$, then we say that $v$ is normalized. 

Let $(R,\m)$ be a Noetherian local domain of $\dim d$ and $R^{+}$ be its integral closure in an algebraic closure of its fraction field where $\m^+$ be its unique maximal ideal. Let $R^{+}$ be equipped with a normalized value map. Such normalized value map can arise in the following way: 

Let $K$ be the field of fraction of $R$. For the maximal ideal $\m$, we can get a valuation ring $(R_v, \m_v)$ of $K$ with maximal ideal $\m_v$ such that $R\subset R_v$ and $\m_v\cap R=\m$. Thus we get a valuation $v$ which is positive on $\m$ and zero on the units of $R$. Let $K^+$ be the algebraic closure of $K$. Now consider a valuation on $K^+$ and call it $v$ again, which extends the valuation of $R$ as well as that of $K$. If $(V,\m_V)$ is the valuation ring for $v$ in $K^+$, then from a well known result of valuation theory we get that $m_V\cap R^+= m^+$. Thus valuation on every element in $m^+$ is positive and that for every unit of $R^+$ is zero. So, we get a normalized valuation on $R^+$. This whole construction can be done for any prime ideal $\p$ of $R$ since for everyprime ideal $\p$ of $R$ there exists a valuation ring $(R_{v'},\m_{v'})$ of $K$ such that $R\subset R_{v'}$ and $\m_{v'}\cap R=\p$. 

Consider a local $R^{+}$-algebra $A$. We recall the following definition of a closure operation from first paragraph of Section 5 of \cite{AB}, which extend the notion of dagger closure in \cite{HH}, via a local algebra over $R^+$.

\begin{definition}\label{def:almost closure of an ideal}
Let $A$ be a local $R^{+}$-algebra where $R^{+}$ is equipped with a normalized valuation $v:R^{+}\rightarrow \RR\cup\{\infty\}$ and $M$ be an $A$-module. Consider a submodule $N\subset M$. Let $x\in M$. Then we say whenever $x\in N^{v}_M$ if and only if for every $\epsilon> 0$, there exists $a\in R^{+}$ such that $v(a)< \epsilon$ and $ax\in N$ or sometimes we can say there exists $a\in R^{+}$ with arbitrarily small $v(a)$ and $ax\in N$. A submodule $N$ is called $v$-closed if $N= N^{v}_M$.
\end{definition}

We recall the definition of almost zero module.

\begin{definition}
For a normalized valuation $v$ on $R^+$, we define an $A$-module $M$ as an almost zero module, if for every $m\in M$ and for every $\epsilon>0$ there exists $a\in A$, such that $v(a)<\epsilon$ and $am= 0$. 
\end{definition}

Throughout this work we fix $v$. We recall the definition of almost Cohen-Macaulay $R^{+}$-algebra and almost regular sequence \cite{Ro}. %

\begin{definition}
Let $R^{+}$ be equipped with a valuation $v$. An $R^{+}$-algebra is almost Cohen-Macaulay if for every system of parameters $x_1,\ldots, x_d$, $((x_1,\ldots, x_i):x_{i+1})/(x_1,\ldots, x_i)$ is almost zero $R^{+}$-module for every $i=1\to d-1$ and $A/\m A$ is not almost zero. If we do not put any restriction on $A/\m A$, we call it an weakly almost Cohen-Macaulay \cite{Sh}. In the above case we say $x_1,\ldots,x_d$ forms an almost regular sequence.
\end{definition}

In this context, we recapitulate the definition of a big Cohen-Macaulay Algebra and balanced big Cohen-Macaulay algebra as stated in the begining of the Introduction: Here we give an example of big Cohen-Macaulay algebra which is not balanced (see page 342 of \cite{BH}):  

Let $R=k[[x,y]]$ be a power series ring of two indeterminates  over a field $k$. Consider an $R$-algebra $B=R \times Q$ where $Q$ is the field of fractions of $R/yR$. Here we have the natural $R$-algebra map $R\rightarrow B$ which sends $r\in R$ to $(r, \bar{r})\in B$ where $\bar{r}$ is the image of $r$ in $Q$. Clearly $(x,y)$ forms a $B$-regular sequence but $(y,x)$ does not.  For a balanced big Cohen-Macaulay algebra, since every system of parameters is a regular sequence so is every permutation of system of parameters. Thus $B$ in this example is a big Cohen-Macaulay algebra but it is not balanced big Cohen-Macaulay.

Since regular sequence is a trivial example of almost regular sequence, in this context $B$ also serves as an example of an almost Cohen-Macaulay algebra where not every system of parameters forms an almost regular sequence except some of them. 

We list the following properties of the closure as defined in Definition 1. The proof of all these properties is given in \cite{AB}, here we restate it for completeness. 

\begin{prop}
Let $M$, $M'$ be modules over a local $R^{+}$-algebra $A$ with a normalized valuation $v:R^{+}\rightarrow \RR\cup\{\infty\}$. Consider the arbitrary submodules $N$, $W$ of $M$. Then the following are true:\begin{enumerate}
\item[$\mathrm{(i)}$] $N^{v}_M$ is a submodule of $M$ containing $N$.
\item[$\mathrm{(ii)}$] $(N^{v}_M)^{v}_M= N^{v}_M$.
\item[$\mathrm{(iii)}$] If $N\subset W\subset M$, then $N^{v}_M\subset W^{v}_M$.
\item[$\mathrm{(iv)}$]  Let $f:M\to M'$ be a homomorphism. Then $f(N^{v}_M)\subset f(N)^{v}_{M'}$.
\item[$\mathrm{(v)}$] If $N^{v}_M= N$ then $0^{v}_{M/N}= 0$.
\end{enumerate}
In addition to, if $A$ is weakly almost Cohen-Macaulay then following is true:
\begin{enumerate}
\item[$\mathrm{(vi)}$] Let $x_1, \ldots , x_{k+1}$ be a partial system of parameters for $A$, and let $J =
(x_1, \ldots , x_k)A$. Suppose that there exists a surjective homomorphism $f:M\to
A/J$ such that $f(u) = \bar{x}_{k+1}$, where $\bar{x}$ is the image of $x$ in $A/J$. Then
$(Au)^{v}_M\cap ker f \subset (Ju)^{v}_M$.
\end{enumerate}
\end{prop}

\begin{proof}
\begin{enumerate}
\item[$\mathrm{(i)}$] Clearly $N\subset N_{M}^{v}$.
For $x, y\in N^{v}_M$, take $\epsilon> 0$ and choose $a, b\in R^{+}$ such that $v(a), v(b)< \epsilon/2$ and $ax, by\in N$. Thus we have $v(ab)< \epsilon$ and $ab(x+ y)\in N$. Thus $x+y\in N^{v}_M$. Consider $x\in N^{v}_M$ and $b\in A$. Since there exists $a\in R^{+}$ such that $v(a)< \epsilon$ and $ax\in N$ and since $N$ is a submodule, we find $a(bx)\in N$ and $bx\in N^{v}_M$. Thus it is easy to see that $N^{v}_M$ is a submodule containing $N$.
\item[$\mathrm{(ii)}$] Take $x\in (N^{v}_M)^{v}_M$. For $\epsilon> 0$ and choose $a\in R^{+}$ such that $v(a)< \epsilon/2$ and $ax\in (N^{v}_M)$. Similarly, for $\epsilon> 0$ and choose $b\in R^{+}$ such that $v(b)< \epsilon/2$ and $(ba)x\in N$. Thus we find $ba\in R^{+}$ with $v(ba)< \epsilon$ such that $(ba)x\in N$. So $x\in (N^{v}_M)$, which yields the claim.
\item[$\mathrm{(iii)}$] This is easy and we leave it to reader.
\item[$\mathrm{(iv)}$]  Consider $x\in N^{v}_M$, thus for every $\epsilon> 0$, there exists $a\in R^{+}$ such that $v(a)< \epsilon$ and $ax\in N$. This implies $f(ax)= af(x)$ is in $f(N)$ where $a\in R^{+}$ is of arbitrarily small positive order. So $f(x)\in f(N)^{v}_{M'}$.
\item[$\mathrm{(v)}$] Consider $\bar{x}\in 0^{v}_{M/N}$ which is the image of $x$ in $M/N$. This implies that $ax\in N$ for the element $a\in R^{+}$ of arbitrarily small positive order. So $x\in N^{v}_M= N$ and $\bar{x}\in 0$.
\item[$\mathrm{(vi)}$]Take $x\in (Au)^{v}_M\cap \ker f$. For every $\epsilon >0$ there exists $a\in R^{+}$ of $v(a)<\epsilon/2$ such that $ax= bu\in Au$ and $af(x)=0=bf(u)=b\bar{x}_{k+1}$. This implies $bx_{k+1}\in J$ i.e. $b\in (J:x_{k+1})$. Since $A$ is weakly almost Cohen-Macaulay, for every $\epsilon >0$ there exists $c\in R^{+}$ of $v(c)<\epsilon/2$ such that $cb\in J$. Thus for $\epsilon >0$ there exists $ac\in R^{+}$ of $v(ac)<\epsilon$ and $(ac)x= cbu\in Ju$. So $x\in (Ju)^{v}_M$.
\end{enumerate}
\end{proof}

With the help of the closure as defined in Definition 1, we can define a closure operation for submodules of an arbitrary $R$-module $M$. This definition is similar to that, given in the first paragraph of Section 5 of \cite{AB}, where it was given for finitely generated $R$ modules. 

\begin{definition}
Let $(R,\m)$ be a Noetherian local domain and let $A$ be a local $R^{+}$-algebra and $R^{+}$ is equipped with a normalized valuation $v:R^{+}\rightarrow \RR\cup\{\infty\}$. For any $R$-module $M$ and for its submodule $N$ we define submodule $N_{M}^{\bold{v}}$ such that $x\in N_{M}^{\bold{v}}$ if $1\otimes x \in \im(A\otimes N\to A\otimes M)_{A\otimes M}^{v}$. 
\end{definition}

\begin{rem}\label{2}  
The closure defined in Definition 4 depends not only on the value map but also on the $R^{+}$-algebra $A$. So, from now, we fix an $R^+$-algebra $A$ to perform the closure operation. In defining closure operation (see Definition 1 and Definition 4), we choose a local $R^+$-algebra $A$ (with a local map). This choice ensures that we can always have system of parameters inside the maximal ideal of $A$ and $A/\m A$ is nonzero. This choice is necessary since sometimes we need $A$ to be an almost Cohen-Macaulay algebra. 

In defining closure operation, one can also use $R^+$-algebra $A$ which is not necessarily local, but whenever it is needed, we have to put an extra assumptions. 
\end{rem}

We recall the definition of Dietz closure \cite{Di1}, and also big Axioms of Dietz, see definition 1.1 of \cite{Di2}. We observe the following proposition.

\begin{prop}
Let $(R,\m)$ be a Noetherian local domain and let $R^{+}$ be equipped with a normalized valuation $v:R^{+}\rightarrow \RR\cup\{\infty\}$. Let $M$, $M'$ be arbitrary $R$-modules. Consider the submodules $N$, $W$ of $M$. Then the following are true:
\begin{enumerate}
\item[$\mathrm{(i)}$]  $N^{\bold{v}}_M$ is a submodule of $M$ containing $N$.
\item[$\mathrm{(ii)}$]$(N^{\bold{v}}_M)^{\bold{v}}_M= N^{\bold{v}}_M$.
\item[$\mathrm{(iii)}$] If $N\subset W\subset M$, then $N^{\bold{v}}_M\subset W^{\bold{v}}_M$.
\item[$\mathrm{(iv)}$] Let $f:M\to M'$ be a homomorphism. Then $f(N^{\bold{v}}_M)\subset f(N)^{\bold{v}}_{M'}$.
\item[$\mathrm{(v)}$] 
If $N^{\bold{v}}_M= N$ then $0^{\bold{v}}_{M/N}= 0$.
\item[$\mathrm{(vi)}$] 
If $R^{+}\to A$ is injective and $A/\m A$ is not almost zero then $0^{\bold{v}}_R= 0$ and $m^{\bold{v}}_R= m$. In particular for every proper ideal $I\subset R$, $I^{\bold{v}}$ is proper.
\end{enumerate}
In addition to, if $A$ is weakly almost Cohen-Macaulay then following is true:
\begin{enumerate}
\item[$\mathrm{(vii)}$] Let $x_1, \ldots , x_{k+1}$ be a partial system of parameters for $R$, and let $J =
(x_1, \ldots , x_k)R$. Suppose that there exists a surjective homomorphism $f:M\to
R/J$ such that $f(u) = \bar{x}_{k+1}$, where $\bar{x}$ is the image of $x$ in $R/J$. Then
$(Ru)^{\bold{v}}_M\cap ker f \subset (Ju)^{\bold{v}}_M$.
\end{enumerate}

In other words, if $A$ is an almost Cohen-Macaulay $R^{+}$-algebra where $R^{+}$ is contained in it as a subdomain, then closure operation of Definition 4, satisfies big Axioms of Dietz.
\end{prop}

\begin{proof}
\begin{enumerate}
\item[$\mathrm{(i)}$]  Clearly $N\subset N_{M}^{\bold{v}}$. To prove $ N_{M}^{\bold{v}}$ is a submodule, let $x, y\in N_{M}^{\bold{v}}$ and $r\in R$. Then $$x\otimes 1, y\otimes 1 \in \im(A\otimes N\to A\otimes M)_{A\otimes M}^{v}.$$ Thus $(x+y)\otimes 1, rx\otimes 1 \in \im(A\otimes N\to A\otimes M)_{A\otimes M}^{v}$. These yield that $x+y, rx\in N_{M}^{\bold{v}}$.
\item[$\mathrm{(ii)}$] Let $x\in (N^{\bold{v}}_M)^{\bold{v}}_M$. This implies $x\otimes 1 \in \im(A\otimes N_{M}^{\bold{v}}\to A\otimes M)_{A\otimes M}^{v}$. Equivalently, for $\epsilon >0$ there exists $a\in R^{+}$ such that $v(a)< \epsilon/2$ and $a\otimes x= \sum_{i=1}^{l}a_i\otimes x_i$ where $x_i\in N^{\bold{v}}_M$ and  $a_i\in A$. Choose $c_i \in R^{+}$ such that for every $i$, $v(c_i)< \epsilon/2l$ and $x_i\otimes c_i \in \im(A\otimes N\to A\otimes M)$. Take $c=\prod_{i=1}^{l}c_i$. Thus $x\otimes ac \in \im(A\otimes N\to A\otimes M)$ and $x\in (N^{\bold{v}}_M)$.
\item[$\mathrm{(iii)}$] This is trivial.
\item[$\mathrm{(iv)}$] Let $x\in N^{\bold{v}}_M$ and this gives $a\otimes x= \sum_{i=1}^{l}a_i\otimes x_i$ for element $a\in R^{+}$ of arbitrarily small order, where $a_i\in A$ and $x_i\in N$. Applying $f\otimes1_A$, we get $a\otimes f(x)= \sum_{i=1}^{l}a_i\otimes f(x_i)$ and this finishes the proof.
\item[$\mathrm{(v)}$] Denote  the image of $x\in M$ in $M/N$ by $\bar{x}$. If $\bar{x}\in 0^{\bold{v}}_{M/N}$ then $a\otimes \bar{x}\in 0$ and this implies $a\otimes x\in \im(A\otimes N\to A\otimes M)$ for element $a\in R^{+}$ of arbitrarily small order. Thus $x\in N^{\bold{v}}_M= N$ and $\bar{x}\in 0$.
\item[$\mathrm{(vi)}$]
Due to (e) of Lemma 1.3 of \cite{Di2}, the condition $0^{\bold{v}}= 0$ follows. For $m^{\bold{v}}= m$, suppose on the contrary that $s\in \fm^{\bold{v}}$ for some unit element $s$. It turns out that $sc\in \fm A$ for every element $c$ of arbitrarily small order, i.e.,  $\fm A$ contain elements of arbitrarily small order, since $s$ is a unit. This provides a contradiction, because $A/mA$ is not almost zero.
\item[$\mathrm{(vii)}$] Take $x\in (Ru)^{\bold{v}}_M\cap \ker f$. For every $\epsilon >0$ there exists $a\in R^{+}$ of $v(a)<\epsilon/2$ such that $a\otimes x= \sum_{i=1}^{n}a_i\otimes r_i u$. Since $f(x)=0$, $f(x)\otimes a= x_{k+1}(\sum_{i=1}^{n}r_ia_i)+JA=0$. This implies $\sum_{i=1}^{n}r_ia_i\in (JA:x_{k+1}A)\subset (JA)^{v}$. So there exists $c\in R^{+}$ with $v(c)< \epsilon/2$ such that $c\sum_{i=1}^{n}r_ia_i\in JA$. Thus $$ac\otimes x= u\sum_{i=1}^{n}r_ia_i\otimes c\in \im(JA\otimes Ru \to A\otimes M)$$ and this gives $ac\otimes x \in \im(A\otimes Ju \to A\otimes M)$ with $v(ac)< \epsilon$. So, we finish the proof.
\end{enumerate}
\end{proof}

In Proposition 1.4 of \cite{Di1}, we have the definition of colon capturing property of a closure. Here in (vii) of Propositiom 2.2 we observe that $A$ has colon capturing property. Now, as a consequence of above proposition we have the following corollary where we observe that the condition $0^{\bold{v}}= 0$ is equivalent to the condition that $R^+$ should be inside of $A$ as a subdomain. 

\begin{cor}
Let $(R,m)$ be a Noetherian local domain and let $R^{+}$ be equipped with a valuation $v$. Consider a local $R^{+}$-algebra $A$ and both the closure operations of Definition 1 and Definition 4. Then, closure operations satisfy all the properties given in Proposition 2.1 and Proposition 2.2 if and only if $A$ is an almost Cohen-Macaulay $R^{+}$-algebra where $R^{+}$ is contained in it as a subdomain.
\end{cor}

\begin{proof}
From Proposition 2.1 and Proposition 2.2 `if' follows. Here, we prove `only if'. Proof of (vi) of Proposition 2.2 implies that if $m^{\bold{v}}= m$ then $A/mA$ is not almost zero. Moreover, $0^{\bold{v}}= 0$ also implies $R^{+}\to A$ is injective: take $0\neq a\in R^{+}$ such that its image in $A$ is zero. Since $a$ is integral over $R$, there exists a minimal monic expression $a^n+ r_1a^{n-1}+\ldots +r_n= 0$, where each $r_i\in R$ with $r_n\neq 0$. Take the image of the expression in $A$ which gives image of $r_n$ is zero in $A$. So $r_n\in 0^{\bold{v}}= 0$. This gives $r_n= 0$. So we arrive at a contradiction and $a$ is zero in $R^{+}$. 

Finally, let $x_1, \ldots , x_{k+1}$ be a part system of parameters for $A$, and let $J =(x_1, \ldots , x_k)A$. Consider (vi) of Proposition 2.1 and we choose $M$, $u\in M$ there, as $A$ and $1\in A$ here. Moreover, we consider the surjective $A$-linear map $f:A\to
A/J$ such that $f(1) = \bar{x}_{k+1}$, where $\bar{x}$ is the image of $x$ in $A/J$. Now $x\in (J:x_{k+1})$ if and only if $x\in ker f$. Since $A^{{v}}_A=A$, from the result of (vi) of Proposition 2.1, we get $x\in (J)^{{v}}_A$. Thus, we finish the proof of this corollary.
\end{proof}

We conclude the section with a brief exposition of phantom extension (see \cite{HH3}, \cite{Di1} and \cite{Di2}). Let $R$ be a ring with a closure operation $\cl$, $M$ an arbitrary $R$-module, and $\alpha:R\rightarrow M$ an injective map with cokernel $Q$. We have a short exact sequence 

\[\begin{CD}
0 @>>> R @>{\alpha}>> M @>>> Q @>>> {0.}
\end{CD}\]

Let $P_\bullet$ be a projective resolution (equivalently, free resolution since $R$ is local) for $Q$ over $R$. Then, this yields the following commutative diagram 

\[\begin{CD}
0 @>>> R @>{\alpha}>> M @>>> Q @>>> 0 \\
@AAA @AA{\phi}A @AAA @AA{\id}A @. \\
P_2 @>>> P_1 @>{d}>> P_0 @>>> Q @>>> 0. \\
\end{CD}\]

Let $\epsilon \in \Ext_R^1(Q,R)$ be the element corresponding to this short exact sequence via the Yoneda correspondence. We say that $M$ (more precisely $\alpha:R\rightarrow M$) is a phantom extension via closure operation $\cl$, if for above projective resolution $P_\bullet$ of $Q$, a cocycle representing $\epsilon$ in $\Hom_R(P_1,R)$ is in $\im (\Hom_R(P_0,R)\rightarrow \Hom_R(P_1,R))_{\Hom_R(P_1,R)}^{\cl}$. 

\section{almost cohen-macaulay algebras are solid}

We recall the definition of solid module and algebra \cite{Ho}. It is a well-known result due to Hochster (see, Corollary 10.6 of \cite{Ho}) that for a complete local domain, any algebra which maps into big Cohen-Macaulay algebra is solid. Similar is true for almost Cohen-Macaulay algebra.

\begin{thm}
(a) Let $R$ be a complete local domain of mixed characteristic $p>0$ and $S$ be an almost Cohen-Macaulay algebra. Then $S$ is a solid $R$-algebra.\\
(b) In particular, any algebra which maps into an almost Cohen-Macaulay algebra is solid.\\
(c) Let $R^+ = \lim_{\rightarrow} R_\alpha$ be a direct limit of Noetherian local rings $R_\alpha$, where each $R_\alpha$ is module finite over $R$. If $S$ is an almost Cohen-Macaulay $R^+$-algebra, then $S$ is a solid $R_\alpha$-algebra for each $\alpha$.
\end{thm}

\begin{proof}
(a) Let $(R,\m)$ be a complete local domain of dimension $d$ and $S$ be an almost Cohen-Macaulay $R$-algebra. Let $H^d_{\m S} (S)$ be the local cohomology module in the support of $\m$ and we can view it as $\lim_{\rightarrow} S/(x^t_1,\ldots,x^t_d)S$ with multiplication map $"x_1\ldots x_d"$. We would like to show $\lim_{\rightarrow} S/(x^t_1,\ldots,x^t_d)S\neq 0$. From definition $S/(x_1,\ldots,x_d)S$ is nonzero. If $\lim_{\rightarrow} S/(x^t_1,\ldots,x^t_d)S= 0$, then image of 1 in some $S/(x^{t+1}_1,\ldots,x^{t+1}_d)S$ will be zero. This implies $x^t_1\ldots x^t_d\in (x^{t+1}_1,\ldots,x^{t+1}_d)S$. Set $x^t_1\ldots x^t_d= s_1x^{t+1}_1+\ldots+ s_dx^{t+1}_d$. This gives $x^t_1(x^t_2 \ldots x^t_d- s_1x_1) \in (x^{t+1}_2,\ldots,x^{t+1}_d)S$. Since $S$ is almost Cohen-Macaulay, $x^t_1, x^{t+1}_2\ldots,x^{t+1}_d$ are almost regular sequence. So, for $\epsilon/d>0$ there exists $b_1\in R^+$ such that $v(b_1)< \epsilon/d$ and $b_1(x^t_2 \ldots x^t_d- s_1x_1) \in (x^{t+1}_2,\ldots,x^{t+1}_d)S$ and this yields $b_1(x^t_2 \ldots x^t_d) \in (x_1, x^{t+1}_2,\ldots,x^{t+1}_d)S$. Now, we repeat the process for $x_2$ and in a similar way for $\epsilon/d>0$ we can get $b_2\in R^+$ such that $v(b_2)< \epsilon/d$ and $b_1b_2(x^t_3 \ldots x^t_d) \in (x_1, x_2, x^{t+1}_3,\ldots,x^{t+1}_d)S$. Thus, repeating this process for $d$ times we get $b= b_1\ldots b_d\in (x_1,\ldots,x_d)S$ with $v(d)<\epsilon$ and we get that $(x_1,\ldots,x_d)S$ as well as $\m S$ contains elements of arbitrarily small order. This implies that $S/\m S$ is almost zero. This is a contradiction and thus we prove $H^d_{\m S} (S)$ is non-zero. Now using Corollary 2.4 of \cite{Ho} we can prove that $S$ is a solid $R$-algebra. 

(b) This is straight forward. 

(c) Since for each index $\alpha$, maximal ideal $\m_\alpha$ and $\m R_\alpha$ are $m_\alpha$-primary, $H^d_{\m S} (S)= H^d_{\m_\alpha S} (S)$ and the assertion follows.
\end{proof}

\section{almost cohen-macaulay algebras map into balanced big cohen-macaulay modules}

In this section, we study when an almost Cohen-Macaulay algebra maps to balanced big Cohen-Macaulay module. We begin the section with the following corollary which is a consequence of Theorem 3.1 of the previous section.

\begin{cor}
Let $R$ be a complete local domain with $F$-finite residue field and $S$ be an almost Cohen-Macaulay $R^+$-algebra. Then, there always exists a balanced big Cohen-Macaulay module over some complete local domain $R'$ (it may not be a balanced big Cohen-Macaulay $R$-module), such that $R'$ is homomorphic image of $R$ and $S$ is mapped into that balanced big Cohen-Macaulay module.
\end{cor}

\begin{proof}
From Theorem 3.1, we have that $S$ is solid. Passing to the ring $R/pR$, we have $S/pS$ is solid as $R/pR$-algebra, which can be also viewed as $(R/pR)^+$-algebra. Being complete, $R/pR$ is $F$-finite, since from the hypothesis, residue field of $R/pR$, which is also the residue field of $R$, is $F$-finite. Now using Corollary 4.9 of \cite{Di2}, $S/pS$ can be modified through a sequence of maps into a balanced big Cohen-Macaulay $R/pR$-module. So $S$ is mapped into that module. 
\end{proof}

\begin{rem}
For positive integers $i$, consider the sets $\{\epsilon_i: \epsilon_i\in \RR\}$ and $\{c_i: c_i\in R^+\}$ such that $\lim_{i\rightarrow\infty} \epsilon_i= 0$ and $v(c_i)< \epsilon_i$. Let $R^+ = \lim_{\rightarrow} R_\alpha$ be a direct limit of Noetherian local rings $R_\alpha$, where each $R_\alpha$ is module finite over $R$. If $c_i\in R_\alpha$ for some $\alpha$, then we call that $R_\alpha$ as $R_i$ and it is always possible to choose $R_i$'s from the directed set of Noetherian local rings such that for $i< j$, $R_i\subset R_j$. From above Proposition 3.1, it is clear that for every positive integer $i$, almost Cohen-Macaulay algebra $S$ is a solid $R_i$ algebra. 
\end{rem}

We propose a special kind of almost Cohen-Macaulay algebra.

\begin{definition}
Consider the situation of Remark 2. We call an almost Cohen-Maculay $R^+$-algebra $S$ satisfies (*) if for every positive integer $i$, there exists an $R$-homomorphism $h_i$ from $S$ to $R_i$, such that it sends 1 to $c_i$. 
\end{definition}

\begin{thm} 
Let $(R,\m)$ be a local domain of mixed characteristic $p>0$, and let $S$ be an almost Cohen-Macaulay algebra which satisfies (*). Then $S$ is a phantom extension of $R$ via the closure defined in Definition 4.
\end{thm}

\begin{proof}
From Theorem 3.1, $S$ is a solid $R$-algebra. Since $R$ is a domain, $R/I$ can not be a solid algebra for any proper ideal $I$ and so $\beta: R\rightarrow S$ must be an injection. 
Thus, we have the following diagram, where bottom row is the projective (equivalently free) resolution of $Q$:

\[
\CD
0 @>>>R@>\beta>>S@>>>Q@>>>0 @.\\
@. @A{f} AA @AA{g} A @AAI A @. \\
\ldots @>>>G@>\nu>>F@>>>Q@>>>0 @.
\endCD
\]

Consider the sets $\{\epsilon_i: \epsilon_i\in \RR\}$ and $\{c_i: c_i\in R^+\}$ such that $\lim_{i\rightarrow\infty} \epsilon_i= 0$ and $v(c_i)< \epsilon_i$. Now for every integer $n$, consider $R$ algebra $R_n$, which is module finite over $R$ and tensoring with last short exact sequence, again we get a short exact sequence $0\rightarrow R_n\rightarrow R_n\otimes_R S\rightarrow R_n\otimes_R Q\rightarrow 0$. Since by Proposition 2.1 (e) of \cite{Ho}, $R_n\otimes_R S$ is solid over $R_n$ and for any proper ideal $J\subset R_n$, $R_n/J$ is not solid.

According to the hypothesis, for every positive integer $n$, let $h_n : S\rightarrow R_n$ be the $R$-linear map such that $h_n (1)= c_n\neq 0$, with $v(c_n)<\epsilon_n$. Thus, we can define an $R$-linear map $\gamma_n: R_n \otimes S\rightarrow R_n$, given by $\gamma_n (r'\otimes s) = r'h_n (s)$. Clearly, $\gamma_n (r'\otimes 1) = r'c_n$

So, for every $\epsilon > 0$, there exists an element $c_n\in R^+$, such that $v(c_n)< \epsilon_n <\epsilon$ and we have the following diagram 
\[
\CD
0 @>>>R_n@>1\otimes \beta>>R_n\otimes_R S@>>>R_n\otimes_R Q@>>>0 @.\\
@. @A{1\otimes f} AA @AA{1\otimes g} A @AAI A @. \\
\ldots @>>>R_n\otimes_R G@>1\otimes\nu>>R_n\otimes_R F@>>>R_n\otimes_R Q@>>>0 @.
\endCD
\]

Set $\Hom_R(F,R)= F^V$ and $Im(\Hom_{R} (F,R)\rightarrow \Hom_{R} (G,R))= Im (\nu^V)$. Now, $c_n (1\otimes f) = \gamma_n\circ (1\otimes \beta)\circ (1\otimes f) = \gamma_n\circ(1\otimes g)\circ(1\otimes \nu) = (1\otimes \nu)^{V}(\gamma_n \circ(1\otimes g))$
and so $c_n (1\otimes f)\in Im((1\otimes \nu)^{V})$ for all $n$. In other words, $c_n(1\otimes f) \in Im(\Hom_{R_n} (R_n\otimes_R F, R_n)\rightarrow \Hom_{R_n}(R_n\otimes_R G, R_n))$

By Lemma 4.3 of \cite{Di2} we get that $c_n(1\otimes f)\in Im(R_n \otimes_R \Hom_R(F,R)\rightarrow R_n \otimes_R \Hom_R(G,R))$. This implies for each $n$, $c_n(1\otimes f)\in Im(A \otimes_R \Hom_R(F,R)\rightarrow A \otimes_R \Hom_R(G,R))$ via base change, since $A$ is also an $R_n$-algebra. From above, we get for each $n$, $v(c_n) < \epsilon_n$, and that the $\epsilon_n$ form a sequence of real numbers converging to zero. Thus, for all $\epsilon > 0$, there exists an $n$ such that $v(c_n) < \epsilon$ with $c_n(1\otimes f)\in Im(A \otimes_R \Hom_R(F,R)\rightarrow A \otimes_R \Hom_R(G,R))$. Hence from Definition 4, $f\in (Im(\nu^V))_{G^V}^{\bold{v}}$. From Proposition 2.2, we get that the almost Cohen-Macaulay algebra $S$ is a phantom extension of $R$ via closure operation defined in Definition 4. This finishes the proof of the theorem.
\end{proof}

\begin{cor}
Let $(R,\m)$ be a complete local domain of mixed characteristic $p> 0$. Let $S$ be an almost Cohen-Macaulay $R$-algebra which satisfies (*) condition, then $S$ can be modified into a balanced big Cohen-Macaulay module over $R$.
\end{cor}

\begin{proof}
Using Theorem 3.7 of \cite{Di2}, Proposition 2.2 and Theorem 4.2, the result follows.
\end{proof}

\begin{rem}
If there exists a $\psi\in \Hom_R(S,R)$ which sends $1\mapsto u$ for some unit $u\in R$, then we can take $c_i= u$ for every $i$, since $v(u)= 0$ and thus in this case $S$ satisfies (*). Now, for unit $u\in R$, $\psi(1)=u$ is equivalent to the situation that $\psi|R$ is an $R$-linear isomorphism, since any nonzero $R$-linear map from $R$ to $R$ is always injective.
\end{rem}

\begin{acknowledgement}
I would like thank Geoffrey D. Dietz for careful reading of the paper and for all invaluable comments and suggestions for improvement of the paper. I would like to thank the referee for careful reading of the paper and for all the rectifications and modifications in the results of the paper. I am also grateful to the referee for numerous suggestions regarding examples and constructions for the betterment of the paper. 
\end{acknowledgement}

\end{document}